\documentclass{amsart}
\usepackage{etex}
\usepackage{graphicx}
\usepackage{bm,caption,amsbsy,enumitem,amsmath,amsthm,
amssymb,mathtools,amsfonts,multirow,verbatim,tikz,tikz-cd,thmtools,thm-restate,xr,tensor,xcolor}
\usepackage[alphabetic]{amsrefs}
\usetikzlibrary{matrix,arrows,decorations.pathmorphing}
\usepackage[top=1.25in, bottom=1in, left=1.25in, right=1.25in,marginpar=1in]{geometry}
\usepackage{url}
\DefineSimpleKey{bib}{myurl}
\usepackage[hidelinks]{hyperref}
\usepackage{chngcntr}
\usepackage[all]{xy}
\counterwithin{figure}{section}

\newcommand\myurl[1]{\url{#1}}

\BibSpec{webpage}{%
  +{}{\PrintAuthors} {author}
  +{,}{ \textit} {title}
  +{}{ \parenthesize} {date}
  +{,}{ \myurl} {myurl}
  +{,}{ } {note}
  +{.}{ } {transition}
}

\newtheorem{thm}{Theorem}[section]
\newtheorem{prop}[thm]{Proposition}

\newtheorem{cor}[thm]{Corollary}
\newtheorem{lem}[thm]{Lemma}

\theoremstyle{definition}
\newtheorem{define}[thm]{Definition}

\theoremstyle{remark}
\newtheorem{rem}[thm]{Remark}
\newtheorem{example}[thm]{Example}

\newcommand{\R}{\mathbb{R}}
\newcommand{\F}{\mathbb{F}}
\newcommand{\Z}{\mathbb{Z}}

\renewcommand{\d}{\partial}
\renewcommand{\subset}{\subseteq}

\renewcommand{\hat}{\widehat}

\DeclareMathOperator{\cotr}{{cotr}}

\DeclareMathOperator{\Diff}{{Diff}}

\DeclareMathOperator{\SO}{{SO}}

\DeclareMathOperator{\Hom}{{Hom}}
\DeclareMathOperator{\Id}{{Id}}
\DeclareMathOperator{\id}{{id}}
\DeclareMathOperator{\ind}{{ind}}
\DeclareMathOperator{\Int}{{int}}

\DeclareMathOperator{\rk}{{rk}}

\newcommand{\bF}{\mathbb{F}}

\newcommand{\cC}{\mathcal{C}}

\newcommand{\cF}{\mathcal{F}}

\newcommand{\SFH}{\mathit{SFH}}

\newcommand{\HFKh}{\widehat{\mathit{HFK}}}
\newcommand{\HFK}{\mathit{HFK}}

\newcommand{\g}{\gamma}

\newcommand{\uv}{\underline{v}}

\title{Distinguishing slice disks using knot Floer homology}

\author{Andr\'as Juh\'asz}%
\address{Mathematical Institute, University of Oxford, Andrew Wiles Building,
Radcliffe Observatory Quarter, Woodstock Road, Oxford, OX2 6GG, UK}%
\email{juhasza@maths.ox.ac.uk}%

\author{Ian Zemke}
\address{Department of Mathematics\\Princeton University\\  Princeton, NJ 08544, USA}
\email{izemke@math.princeton.edu}
\begin{document}

\begin{abstract}
  We study the classification of slice disks of knots up to isotopy and diffeomorphism
  using an invariant in knot Floer homology. We compute the invariant of a slice disk obtained by
  deform-spinning, and show that it can be effectively used to distinguish non-isotopic slice disks
  with diffeomorphic complements.

  Given a slice disk of a composite knot, we define a numerical stable diffeomorphism invariant
  called the rank. This can be used to show that a slice disk is not a boundary connected sum,
  and to give lower bounds on the complexity of certain hyperplane sections of the slice disk.
\end{abstract}

\maketitle

\section{Introduction}

In this paper, we consider the classification of smooth slice disks in $D^4$ for a knot $K$ in $S^3$.
Even though the existence problem of slice disks, and the closely related slice-ribbon conjecture of Fox
are in the center of research in low-dimensional topology,
this natural question has been little studied in the literature.

We introduce several notions of equivalence: ambient isotopy fixing $S^3$ pointwise, diffeomorphism,
and stable versions of these where one can take connected sums with 2-knots.
Our interest lies in exploring the potential of an invariant of slice disks in knot Floer homology defined
by Marengon and the first author~\cite{JMConcordance}. Until the availability of the
powerful techniques of \cite{JuhaszZemkeContactHandles}, which provide formulas for the trace and cotrace maps in the
link Floer TQFT, computing the invariant of any non-trivial slice disk was beyond reach.
We now provide formulas for infinite families of slice disks,
and show that the invariant can effectively distinguish slice disks up to both stable isotopy and
stable diffeomorphism. In some of these cases, as we shall see,
the fundamental group provides a simpler method of classification up to isotopy,
but it seems unlikely these methods would easily extend to the stable diffeomorphism classification.
The parallel between the fundamental group and Heegaard Floer techniques
in this paper naturally raises the question whether there is a relationship between
the fundamental group and the link cobordism maps.

Given a slice disk $D$ for $K$, one can always produce infinitely many other slice disks by connected summing
$D$ with different 2-knots. We first focus on a less trivial construction, which does not change the
diffeomorphism type of the slice disk complement.
This consists of taking an isotopy from $K$ to itself, and attaching its trace to $D$.

Let $d$ be an automorphism of $(S^3, K)$ that is the identity in a neighborhood $B$ of a point of $K$.
Then the deform-spinning operation of Litherland~\cite{Litherland} -- which is a common generalization
of twist-spinning, due to Zeeman~\cite{twisting}, and roll-spinning, due to Fox~\cite{rolling} --
gives a slice disk $D_{K, d}$ of $-K \# K$, where $-K$ stands for $(-S^3, -K)$; see Section~\ref{sec:spun}.
According to Proposition~\ref{prop:complement}, the diffeomorphism types
of the pairs $(D^4, D_{K, d})$ all coincide. However, by Proposition~\ref{prop:isot},
the isotopy class of $D_{K, d}$ is uniquely determined by the isotopy class of $d$
in $\Diff(S^3, K, B)$. This, in turn, is determined by the action of $d$
on $\pi_1(S^3 \setminus K, p)$ for $p \in B \setminus K$ by the works of Waldhausen~\cite{Waldhausen},
Cerf~\cite{Cerf2}, and Hatcher~\cite{Hatcher}.

Given a slice disk $D$ for a decorated knot $(K, P)$ in $S^3$, Marengon and the first author~\cite{JMConcordance}
defined a non-zero element
\[
t_{D,P} \in \HFKh(K,P)
\]
that is invariant under isotopies of~$D$.
We will review the relevant definitions in Section~\ref{sec:invariant}.
This invariant is unchanged by taking the connected sum of $D$ with a 2-knot.
So far, there has been no computation of $t_{D, P}$ for a non-trivial slice disk $D$.
The difficulty lies in the fact that it is hard to give a natural construction of $\HFKh(K, P)$,
and hence to distinguish non-zero elements; see~\cite{JTNaturality}.
However, the knot Floer homology of a composite knot has additional structure.

Let us write $V := \HFKh(K, P)$. Our first main result is Theorem~\ref{thm:spinning},
which states that, under a suitable identification
\[
\HFKh(-K \# K, P) \cong V^* \otimes V \cong \Hom(V,V),
\]
the element $t_{D_{K, d}, P}$ corresponds to $d_* \in \Hom(V,V)$.
If $r$ is the diffeomorphism giving rise to roll-spinning, then this implies that the invariants $t_{D_{K, r^l}, P}$
distinguish the slice disks $D_{K, r^l}$ for $l \in \Z$ up to stable isotopy,
even though they have diffeomorphic complements, for infinitely many knots $K$; see Theorem~\ref{thm:rolling}.
(Note that our proof of this result assumes that the formula of Sarkar~\cite{SarkarMovingBasepoints}*{Theorem~1.1}
for the basepoint moving map holds over~$\Z$. The formula over~$\F_2$ only allows us to distinguish
$D_{K, r^l}$ for even and odd~$l$.)
In particular, we obtain an affirmative answer to \cite{JMConcordance}*{Question~1.4},
showing that there is a decorated knot $(K, P)$ with two slice disks $D$ and $D'$ such that $t_{D,P} \neq t_{D',P}$.

More generally, given a decorated concordance $\cC$ from $(K, P)$ to $(K', P')$, one can construct
a slice disk $D_\cC$ for $-K \# K'$. In fact, every slice disk of $-K \# K'$ arises from this construction.
If we write $V = \HFKh(K, P)$ and $V' = \HFKh(K', P')$,
then there is an isomorphism $\HFKh(-K \# K', P) \cong V^* \otimes V'$
that depends on the connected sum sphere $S$,
under which $t_{D_\cC}$ corresponds to the concordance map $F_\cC \in \Hom(V, V')$.
We write $\rk_S(D) = \rk(F_\cC)$.
If the slice disk $D$ of $-K \# K'$ is the boundary connected sum of slice
disks of $K$ and $K'$, then $\rk_S(D) = 1$.

If $K$ and $K'$ are prime knots such that $K' \neq -K$, then the connected sum sphere $S$ is unique,
and the rank of $t_{D, P}$ in $\Hom(V, V')$
is an invariant of $D$ up to \emph{stable diffeomorphism}. We denote this by $\rk(D)$.
Since $t_{D, P}$ preserves the Alexander and Maslov gradings, $\rk(D)$ has refinements
$\rk_j(D, i)$ and $\rk(D, i)$ for $i$, $j \in \Z$.


The rank of a slice disk $D$ of a composite knot gives a lower bound on the complexity
of certain hyperplane sections of $D$. More concretely, if $H$ is a properly embedded 3-ball in $D^4$
transverse to $D$, whose boundary is the connected sum sphere $S$, then we can cap off $(H, D \cap H)$
with the trivial tangle $(D^3, D^1)$ to obtain a link $L$ in $S^3$. Then
\[
\rk_S(D) \le \rk\left(\HFKh(L)\right);
\]
see Theorem~\ref{thm:section}. Furthermore, when $L$ is a knot,
\[
\max \{\, i \in \Z : \rk_S(D , i) \neq 0 \,\} \le g(L).
\]

So far, we have focused on slice disks of composite knots.
We would like to remark that the invariant $t_{D,P}$ can also be used to distinguish
slice disks of \emph{prime} knots. Kim~\cite{Kim} has shown that every knot $K$ admits an invertible
concordance $C$ to a prime knot $K'$, obtained by taking a certain satellite of $K$.
Let $P$ and $P'$ be decorations on $K$ and $K'$, respectively, choose a decoration $\sigma$ on $C$
compatible with these, and let $\cC = (C, \sigma)$.
If $D$ and $D'$ are slice disks of $K$ with $t_{D,P} \neq t_{D',P}$, then
$t_{C \cup D, P'} \neq t_{C \cup D', P'}$, since $t_{C \cup D, P'} = F_\cC(t_{D, P})$
and $t_{C \cup D', P'} = F_\cC(t_{D', P})$, and the concordance map $F_\cC$ is injective;
see~\cite{JMConcordance}. In other words, if the invariant distinguishes the slice disks $D$ and $D'$
of a possibly composite knot $K$,
then it also distinguishes the slice disks $C \cup D$ and $C \cup D'$ of the prime knot $K'$,
up to stable isotopy.

\subsection{Acknowledgements}
We would like to thank David Gabai and Maggie Miller for helpful conversations.
The first author was supported by a Royal Society Research Fellowship,
and the second author by an NSF Postdoctoral Research Fellowship (DMS-1703685).
This project has received funding from the European Research Council (ERC)
under the European Union's Horizon 2020 research and innovation programme
(grant agreement No 674978).

\section{Equivalences of slice disks}\label{sec:def}

Throughout this paper, every slice disk is assumed to be smooth.

\begin{define}
  Let $K$ be a knot in $S^3$. We say that the slice disks $D$, $D' \subset D^4$ for $K$ are \emph{isotopic},
  and write $D \sim D'$, if there is an isotopy $F \colon I \times D^4 \to D^4$ such that $F_0 = \Id_{D^4}$,
  $F_t|_{S^3} = \Id_{S^3}$ for every $t \in I$, and $F_1(D) = D'$. The slice disks $D$ and $D'$ are
  \emph{diffeomorphic} if there is an orientation preserving diffeomorphism
  $\varphi \in \Diff_+(D^4)$ such that $\varphi(D) = D'$. Finally, we say that $D$ and $D'$
  are \emph{stably isotopic/diffeomorphic} if they become isotopic/diffeomorphic after
  taking the connected sum of $D$ and $D'$ with collections of 2-knots.
\end{define}

\begin{lem}\label{lem:diff}
  Suppose that $D$ and $D'$ are slice disks of a knot $K$ in $S^3$.
  \begin{enumerate}
  \item\label{it:1} If there is a diffeomorphism $\varphi \in \Diff(D^4)$ such that $\varphi(D) = D'$
  and $\varphi|_{S^3} = \Id_{S^3}$, then $D$ and $D'$ are isotopic.
  \item\label{it:2} If $D$ and $D'$ are diffeomorphic, then they become
  isotopic by gluing the trace of an isotopy of $K$ in $S^3$ to $D'$.
  \end{enumerate}
\end{lem}

\begin{proof}
  First, consider \eqref{it:1}. As $\varphi$ fixes $S^3$ pointwise, we can assume that it also fixes
  a collar neighborhood $N(S^3)$ of $S^3$ in $D^4$. We can isotope $D$ to a slice disk $D_1$ of $K$
  that lies in $N(S^3)$. Then
  \[
  D \sim D_1 = \varphi(D_1) \sim \varphi(D) = D',
  \]
  as claimed.

  Now suppose that $D$ and $D'$ are diffeomorphic. Then $\varphi(D) = D'$ for some $\varphi \in \Diff_+(D^4)$.
  Since $\Diff_+(S^3)$ is connected, there is an isotopy $\psi \colon I \times S^3 \to S^3$
  such that $\psi_0 = \Id_{S^3}$ and $\psi_1 = \varphi|_{S^3}$.
  If we glue the collars $(I \times S^3, I \times K)$ to $(D^4, D)$ and
  $(I \times S^3, \bigcup_{t \in I} \{t\} \times \psi_t(K))$ to $(D^4, D')$ along $\{1\} \times S^3$,
  we obtain slice disks $\bar{D}$ and $\bar{D}'$ in $D^4 \cup (I \times S^3)$, respectively.
  If we extend $\varphi$ to the collar $I \times S^3$ via $\bar{\varphi}(t, x) = (t, \psi_t(x))$, then
  the extension $\bar{\varphi}$ maps $\bar{D}$ to $\bar{D}'$, and fixes $\{0\} \times S^3$ pointwise.
  Hence $\bar{D}$ and $\bar{D}'$ are isotopic by~\eqref{it:1}.
\end{proof}

\section{Deform-spun slice disks}\label{sec:spun}

Litherland~\cite{Litherland} defined the notion of deform-spinning to construct
2-knots in $\R^4$, generalizing the construction of twist-spinning, due to Zeeman~\cite{twisting},
and roll-spinning, due to Fox~\cite{rolling}. An analogous construction can be used to obtain
slice disks in $D^4$, as follows.

\begin{define}\label{def:spinning}
  Let $a$ be a properly embedded smooth arc in $D^3$.
  Furthermore, let $\phi \colon I \times D^3 \to D^3$ be an isotopy of $D^3$ such that
  $\phi_0 = \Id_{D^3}$, $\phi_t|_{S^2} = \Id_{S^2}$ for every $t \in I$, and $\phi_1(a) = a$.
  Then the \emph{deform-spun} slice disk $D_{a, \phi} \subset D^4$ is defined by taking
  \[
  \bigcup_{t \in I} \{t\} \times \phi_t(a) \subset I \times D^3,
  \]
  and rounding the corners along $\{0, 1\} \times \d D^3$.
  When the arc $a$ is understood, we simply write $D_{\phi}$ instead of $D_{a, \phi}$.
\end{define}

Intuitively, we consider the arc $a$ in $\R^3_-$, which we rotate  about the plane $\R^2$ in $\R^4_+$,
while applying the isotopy $\phi$, until we reach $\R^3_+$.

\begin{prop}\label{prop:complement}
  Let $D_{a, \phi}$ and $D_{a, \psi}$ be deform-spun slice disks. Then $D_{a, \phi}$
  and $D_{a, \psi}$ are diffeomorphic.
\end{prop}

\begin{proof}
  The diffeomorphism is given by $(t, x) \mapsto (t, \psi_t \circ \phi_t^{-1}(x))$
  for $(t, x) \in I \times D^3$.
\end{proof}

If we take $\psi_t$ to be $\Id_{D^3}$ for every $t \in I$,
we obtain that $D^4 \setminus D_{a, \phi}$ is diffeomorphic to
$I \times (D^3 \setminus a)$ for any isotopy $\phi$.

\begin{lem}
  Let $d$ be an automorphism of $(D^3, a)$ such that $d|_{S^2} = \Id_{S^2}$.
  Then there is an isotopy $\phi \colon I \times D^3 \to D^3$
  as in Definition~\ref{def:spinning}, such that $\phi_1 = d$. Furthermore,
  the isotopy class of the deform-spun disk $D_{a, \phi}$ only depends on $d$,
  which we denote by $D_{a, d}$.
\end{lem}

\begin{proof}
By the work of Hatcher~\cite{Hatcher}, the group $\Diff(D^3, S^2)$ is contractible.
Hence, there exists an isotopy $\phi$ with $\phi_1 = d$. Furthermore, $\phi$ is unique up to isotopy.
\end{proof}


\begin{define}
Let $K$ be a knot in $S^3$, and suppose that the open 3-ball $B$ intersects $K$ in an unknotted arc.
Then $(S^3 \setminus B, K \setminus B)$ is diffeomorphic to a ball-arc pair $(D^3, a)$.
Suppose that we are given a diffeomorphism $d \in \Diff(S^3, K)$ that is the identity on $B$.
Then the \emph{deform-spun} slice disk $D_{K, d} \subset B^4$ for $-K \# K$ is defined
to be $D_{a, d|_{S^3 \setminus B}}$.
\end{define}

We now recall the definition of roll-spinning, based on the description of
Litherland~\cite{Litherland}*{Example~2.2}.

\begin{define}\label{def:rolling}
  Let $K$ be a knot in $S^3$. Choose a tubular neighborhood $N(K) \approx K \times D^2$
  of $K$, and let $X = S^3 \setminus \Int(N(K))$ be the knot exterior.
  Furthermore, let $\d X \times I$ be a collar of $\d X$ in $X$.
  We identify $K$ with $\R/\Z$. Choose a smooth monotonic function $\varphi \colon \R \to I$
  such that $\varphi(t) = 0$ for $t \le 0$ and $\varphi(t) = 1$ for $t \ge 1$.
  We define the diffeomorphism $r \colon (S^3, K) \to (S^3, K)$ by the formula
  \[
  r(\bar{x}, \bar{\theta}, t) = \left(\overline{x + \varphi(t)}, \bar{\theta}, t \right)
  \text{ for } (\bar{x}, \bar{\theta}, t) \in K \times \d D^2 \times I \approx \d X \times I,
  \]
  and let $r(p)= p$ for $p \in S^3 \setminus (\d X \times I)$.

  Let $B \subset N(K)$ be an open 3-ball that intersects $K$ in an unknotted arc.
  Then $r$ is the identity on $B$. We define the \emph{$l$-roll-spin} of $K$ to be $D_{K, r^l}$.
\end{define}

We now give an equivalent definition of $D_{K, r^l}$. Let $A_l$ be an arc on the cylinder
$I \times K \subset I \times S^3$ such that $[p_K \circ A_l] = l \in \pi_1(K) \cong \Z$, where
$p_K \colon I \times K \to K$ is the projection. Furthermore, we let $A_0 = I \times \{x\}$.
Suppose that $\nu$ is an $I$-invariant normal framing of $I \times K$ in $I \times S^3$
that restricts to an odd framing of $\{0\} \times K \subset \{0\} \times S^3$.
We endow $A_l$ with the normal framing that is given by the normal of $A_l$ in $I \times K$,
followed by the framing $\nu|_{A_l}$. Note that homotopy classes of normal framings of $A_l$ relative
to $\d A_l$ correspond to $\pi_1(\SO(3)) \cong \Z_2$, hence the above framing does not depend on the choice of~$\nu$.
As we shall see in the proof of Lemma~\ref{lem:rolling},
the framed arcs $A_0$ and $A_l$ are homotopic, hence isotopic in $I \times S^3$ through an ambient isotopy
that fixes $\{0,1\} \times S^3$ pointwise. This induces a diffeomorphism
\[
d_l \colon (I \times S^3) \setminus N(A_l) \to (I \times S^3) \setminus N(A_0) \approx D^4,
\]
which is the identity on $(\{0, 1\} \times S^3) \setminus N(A_l)$, and is given
by the normal framing of $A_l$ on $\d N(A_l) \setminus (\{0, 1\} \times S^3)$.
Note that $d_l$ is only well-defined up to the action of $\Diff(D^4, S^3)$.
For an illustration, see Figure~\ref{fig:1}.

\begin{lem}\label{lem:rolling}
The slice disks $D_{K, r^l}$ and $d_l((I \times K) \setminus N(A_l))$ are isotopic.
\end{lem}

\begin{proof}
Let $\gamma \colon I \to K$ be a smooth orientation-preserving parametrization of
the knot $K$, with $\gamma(0) = \gamma(1) = x$. We take the arc $A_l \colon I \to I \times K$ to be
\[
A_l(t) = (t, \gamma(lt)).
\]
Let $(e_1, e_2, e_3)$ be a positive orthonormal basis of $T_x S^3$
such that $e_1$ is positively tangent to~$K$.
Then there is a loop $\psi \colon I \to \SO(4)$ based at $\id_{\R^4}$ such that,
if we view $\psi(t)$ as an automorphism of $S^3$, then
$\psi(t)(x) = \g(t)$, $d\psi(t)(e_1)$ is a positive tangent to $K$ for every $t \in I$,
and $(d\psi(t)(e_2), d\psi(t)(e_3))$ corresponds to an odd framing of $K$.

We claim that $\psi$ is null-homotopic in $\SO(4)$. Indeed, let $S$ be a Seifert surface for $K$,
and choose tangent vector fields $v_1$ and $v_2$ on $S$ such that $v_1$ is generic,
$v_1|_K$ is tangent to $K$, and $v_2$ is obtained by rotating $v_1$ through an angle of~$\pi/2$ in $TS$.
Let $p_1, \dots, p_k \in \Int(S)$ be the zeros of $v_1$, and let $D_i \subset \Int(S)$ be a small disk about $p_i$.
Furthermore, let $v_3$ be the normal of $S$ in $S^3$, and $v_4(p) = p$ for $p \in S \subset \R^4$.
Then $\uv := (v_1/|v_1|, \dots, v_4/|v_4|)$ maps $S \setminus \{p_1, \dots, p_k\}$ to $\SO(4)$.
The cycle $\uv|_K$ is homologous to $\uv|_{\d D_1 \cup \dots \cup \d D_k}$ in $\SO(4)$,
where the homology is given by $\uv|_{S \setminus \Int(D_1 \cup \dots \cup D_k)}$.
Along $\d D_i$, the vector fields $v_1$ and $v_2$ rotate $\pm 1$, depending on the index of the
zero of $v_1$ at $p_i$, while $v_3$ and $v_4$ are nearly constant if $D_i$ is sufficiently small,
for $i \in \{1, \dots, k\}$. Hence $\uv|_{\d D_i}$ generates $H_1(\SO(4)) \cong \Z_2$.
By the Poincar\'e--Hopf theorem,
\[
\sum_{i = 1}^{k} \ind_{p_i}(v_1) = \chi(S),
\]
which is odd. It follows that $\uv|_K$ generates $H_1(\SO(4)) \cong \pi_1(\SO(4))$.
As the loop $\psi$ differs from $\uv|_K$ by changing $v_2$ and $v_3$ to an odd framing,
it follows that $\psi$ is null-homotopic in $\SO(4)$.

We extend $\psi \colon I \to \SO(4)$ periodically to a map $\psi \colon \R \to \SO(4)$.
We construct a diffeomorphism $\Psi_l \colon I \times S^3 \to I \times S^3$ via
\[
\Psi_l(t, p) := (t, \psi(lt)(p)).
\]
Since $\psi$ is null-homotopic in $\pi_1(\SO(4))$,
the diffeomorphism $\Psi_l$ is isotopic to $\id_{I \times S^3}$ relative to $\{0,1\} \times S^3$.
By construction, $\Psi_l$ maps the framed arc $A_l$ to $A_0$.
Let $B$ be a ball about $x$ so small that $\psi(t)(B)$ intersects $K$
in an unknotted arc for every $t \in I$.
We let $N(A_0) = I \times B$ and $N(A_l) = \Psi_l^{-1}(N(A_0))$,
and set $d_l = \Psi_l|_{(I \times S^3) \setminus N(A_l)}$; see Figure~\ref{fig:1}.

Write $r_t \colon S^3 \to S^3$ for the isotopy of $S^3$ obtained by rotating
$K$ by $2\pi t$ radians and keeping $S^3 \setminus N(K)$ fixed pointwise,
such that $r_0 = \id_{S^3}$ and $r_1 = r$. Note that
$\psi_{lt} \circ r_{lt}$ is an isotopy of $S^3$ from $\Id_{S^3}$ to $r^l$
that fixes the ball $B$ pointwise.
By definition, the slice disk $D_{K, r^l}$ is obtained by
removing $I \times B$ from the trace of the isotopy $\psi_{lt} \circ r_{lt}$.
The isotopy $r_{lt}$ fixes the knot $K$ setwise, so it follows that
the trace of $\psi_{lt} \circ r_{lt}$ is simply the trace of $\psi_{lt}$.
Hence $D_{K, r^l}$ is isotopic to $d_l((I \times K) \setminus N(A_l))$.
\end{proof}


This construction can be generalized, as follows.

\begin{define}\label{def:conc}
Let $K$ and $K'$ be knots in $S^3$ that both pass through a point $x \in S^3$,
and let $(I \times S^3, C, \sigma)$ be a decorated concordance from $K$ to $K'$,
such that one of the components $A$ of $\sigma$ has boundary $\{(0,x), (1,x)\}$.
Then we can obtain a slice disk $D_{\cC}$ for $-K \# K'$ in
\[
(I \times S^3) \setminus N(I \times \{x\}) \approx D^4
\]
by ambient isotoping $A$ to $I \times \{x\}$, with framings as above,
and applying the induced diffeomorphism~$d$ to $C \setminus N(A)$; see Figure~\ref{fig:1}.
As $d$ is well-defined up to $\Diff(D^4, S^3)$, we obtain that $D_{\cC}$ is unique up
to isotopy by Lemma~\ref{lem:diff}.
\end{define}

Conversely, every slice disk of $-K \# K'$ arises from this construction:

\begin{lem}\label{lem:sum}
  Let $K$ and $K'$ be knots in $S^3$ that both pass through a point $x \in S^3$.
  Given a slice disk $D$ of $-K \# K'$, there is
  a decorated concordance $\cC$ from $K$ to $K'$ such that $D = D_\cC$.
\end{lem}

\begin{proof}
  Let us view $D^4$ as $(I \times S^3) \setminus N(I \times \{x\})$, with $K$ in $\{0\} \times S^3$
  and $K'$ in $\{1\} \times S^3$. From the slice disk $D$, we obtain the decorated concordance $\cC = (C, \sigma)$
  by reattaching the 3-handle $N(I \times \{x\})$ to $D^4$, and setting $C = D \cup R$, where $R$ is a band
  attached along the two arcs of $(-K \# K') \cap N(I \times \{x\})$ that contains $I \times \{x\}$.
  Finally, we choose the decoration $\sigma$ such that it consists of two parallel arcs,
  one of which is $I \times \{x\}$.
\end{proof}

\begin{figure}[ht!]
  \centering
  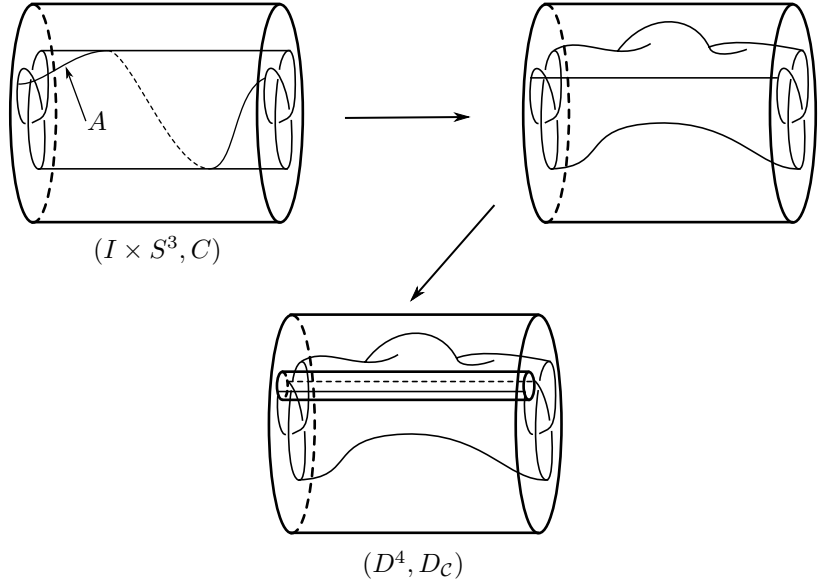
  \caption{Constructing the slice disk $D_{\cC}$ in $D^4$ from a decorated concordance
  $\cC = (C,\sigma)$ and a chosen arc $A$ in $\sigma$.
  \label{fig:1}}
\end{figure}

\begin{define}
  Let $D$ be a slice disk of a knot $K$ in $S^3$. Then we call the homomorphism
  \[
  h_D \colon \pi_1(S^3 \setminus K) \to \pi_1(D^4 \setminus D)
  \]
  induced by the embedding the \emph{peripheral map} of the slice disk $D$.
  Given slice disks $D$ and $D'$ of $K$, we say that $h_D$ and $h_{D'}$ are \emph{equivalent}
  if there is a homomorphism $g \colon \pi_1(D^4 \setminus D) \to \pi_1(D^4 \setminus D')$
  for which $h_{D'} = g \circ h_D$.
\end{define}

If $D$ and $D'$ are isotopic, then clearly their peripheral maps $h_D$ and $h_{D'}$ are equivalent.
Gordon~\cite[Lemma~3.1]{Gordon} proved that if $D$ is a ribbon disk, then $h_D$ is surjective.

\begin{prop}\label{prop:isot}
  Let $D_{K, d}$ and $D_{K, d'}$ be deform-spun slice disks of a knot $-K \# K$ in $S^3$,
  where $d$ and $d'$ are automorphisms of $(S^3, K)$ that are fixed in a neighborhood $B$ of a point on $K$.
  Then $D_{K, d}$ and $D_{K, d'}$ are isotopic if and only if $d$ and $d'$ are isotopic in $\Diff(S^3, K, B)$.
\end{prop}

\begin{proof}
  Pick a basepoint $p \in \d B \setminus K$.
  We are going to write $\pi_1(K)$ for $\pi_1(S^3 \setminus K, p)$
  and $\pi_1(D)$ for $\pi_1(D^4 \setminus D, p)$.
  Choose an isotopy $\phi \colon I \times S^3 \to S^3$
  such that $\phi_0 = \Id_{S^3}$ and $\phi_1 = d$.
  Notice that
  \[
  \pi_1(-K \# K) \cong \pi_1(-K) \ast_\mu \pi_1(K),
  \]
  where the amalgamated product is obtained by identifying
  a meridian $\mu$ of $-K$ with the corresponding meridian of $K$.

  Let $D_{K, \Id}$ be the deform-spun disk corresponding to $\Id_{S^3}$.
  Then consider the diffeomorphism
  \[
  \begin{split}
     \Phi \colon (D^3 \times I) \setminus D_{K, \Id} &\to (D^3 \times I) \setminus D_{K, \phi}. \\
       (x,t) &\mapsto (\phi_t(x), t)
  \end{split}
  \]
  The embedding
  \[
  S^3 \setminus (\Int(B) \cup K) \stackrel{\sim}{\longrightarrow} (D^3 \setminus K) \times \{0\}
  \hookrightarrow (D^3 \times I) \setminus D_{K, \phi}
  \]
  induces an isomorphism on $\pi_1$, whose inverse we denote by $i_\phi \colon \pi_1(D_{K, \phi}) \to \pi_1(K)$.
  Similarly, we define the isomorphism $i_{\Id} \colon \pi_1(D_{K, \Id}) \to \pi_1(K)$.
  With this notation, the following diagram is commutative:
  \[
  \xymatrix{
  \pi_1(-K \# K) \ar[r]^-{h_{D_{K, \Id}}} \ar[d]|-{\Id \ast_\mu d_*} &
  \pi_1(D_{K, \Id}) \ar[r]^-{i_{\Id}} \ar[d]^-{\Phi_*} & \pi_1(K) \ar[d]^-{\rotatebox{90}{=}}\\
  \pi_1(-K \# K) \ar[r]^-{h_{D_{K, \phi}}} & \pi_1(D_{K, \phi}) \ar[r]^-{i_\phi} & \pi_1(K).
  }
  \]
  Here, the map $\pi_1(-K \# K) \to \pi_1(-K \# K)$ takes $\g \ast_\mu \eta$ to $\g \ast_\mu d_*(\eta)$,
  and the vertical arrow $\pi_1(K) \to \pi_1(K)$ is the identity because
  $\Phi|_{(D^3 \setminus K) \times \{0\}} = \phi_0 = \Id$, while
  $\Phi|_{(D^3 \setminus K) \times \{1\}} = \phi_1 = d$. Since
  $i_{\Id} \circ h_{D_{K, \Id}}(\g \ast_\mu \eta) = \g\eta$, we have
  \[
  \begin{split}
     i_\phi \circ h_{D_{K, \phi}} \colon \pi_1(-K \# K)  &\to \pi_1(K). \\
     \g \ast_\mu \eta  &\mapsto \g d_*^{-1}(\eta)
  \end{split}
  \]

  If $D_{K, d}$ and $D_{K, d'}$ are isotopic, their peripheral maps are equivalent.
  So, there exists an isomorphism $g \colon \pi_1(K) \to \pi_1(K)$ such that
  \[
  g(\g d_*^{-1}(\eta)) = \g (d'_*)^{-1}(\eta)
  \]
  for every $\g$, $\eta \in \pi_1(K)$. If we set $\eta = 1$, then we see that $g = \Id_{\pi_1(K)}$.
  Hence, if we set $\g = 1$, we obtain that $d_* = d'_*$.

  Let $X = S^3 \setminus N(K)$ be the knot exterior, and suppose that the basepoint $p \in \d X$.
  Since $\pi_1(\d X) \to \pi_1(X)$ is injective unless $K$ is the unknot,
  $(d|_{\d X})_*$ and $(d'|_{\d X})_*$ agree on $\pi_1(\d X)$. As both $d$ and $d'$ are orientation
  preserving, $d|_{\d X}$ and $d'|_{\d X}$ are isotopic. Hence, we can assume that
  $d$ and $d'$ agree on $\d X$.

  We now construct a homotopy between $d$ and $d'$ that fixes $\d X$ pointwise.
  Choose a triangulation of $X$. Let $T$ be a spanning tree of the 1-skeleton of $X$
  such that $T \cap \d X$ is a spanning tree of the 1-skeleton of $\d X$.
  We can homotope $d|_T$ to $d'|_T$, fixing $T \cap \d X$, since $\pi_1(T, T \cap \d X) = 0$.
  Given any 1-simplex $e$ not in $T$ or $\d X$,
  it determines an element of $\pi_1(X)$. As $d_*(e) = d'_*(e)$, we can homotope
  $d|_e$ to $d'|_e$ relative to $\d e$. Once $d$ and $d'$ agree on the 1-skeleton,
  we can extend the homotopy to the 2- and 3-skeleta as $\pi_2(X) = \pi_3(X) = 0$.

  It now follows from the work of Waldhausen \cite[Theorem~7.1]{Waldhausen}
  that $d$ and $d'$ are isotopic in $\text{Homeo}_+(X, \d X)$,
  and hence also in $\Diff_+(X, \d X)$ according to Cerf~\cite{Cerf2}
  and Hatcher~\cite{Hatcher}. In particular, the isotopy can be chosen to be the identity on $B$,
  so $D_{K, d}$ and $D_{K, d'}$ are isotopic in $\Diff_+(S^3, K, B)$.
\end{proof}

\begin{cor}
  If $K$ is a non-trivial knot in $S^3$, then the roll-spun slice disks $D_{K, r^k}$ and $D_{K, r^{l}}$
  are isotopic if and only if $k = l$.
\end{cor}

\begin{proof}
  Let $\lambda \in \pi_1(K)$ be the homotopy class of the longitude of $K$.
  Then $r_*(\g) = \lambda \g \lambda^{-1}$ for every $\g \in \pi_1(K)$. Since no power of
  $\lambda$ is central in $\pi_1(K)$ for a non-trivial knot~$K$, we see that
  $r^k$ and $r^{l}$ are isotopic in $\Diff(S^3, K , B)$ if and only if $k = l$.
  The result now follows from Proposition~\ref{prop:isot}.
\end{proof}

\section{An invariant of slice disks in knot Floer homology}\label{sec:invariant}

In this section, we review the necessary background on concordance maps and slice disk
invariants in knot Floer homology.

\begin{define}
A \emph{decorated knot} $(K,P)$ in $S^3$ consists of an oriented
knot $K$ in $S^3$, and a pair of basepoints $P = \{w, z\} \subset K$.
A \emph{decorated concordance} $(I \times S^3, C, \sigma)$ from the decorated knot
$(K_0, P_0)$ to $(K_1, P_1)$ consists of an
embedded annulus $C \subset I \times S^3$ such that $C \cap (\{i\} \times S^3) =
\{i\} \times K_i$ for $i \in \{0, 1\}$, together
with a pair of disjoint arcs $\sigma$ on the annulus $A$, both connecting
$K_0 \setminus P_0$ and $K_1 \setminus P_1$.  Furthermore
one component of $C \setminus \sigma$ contains $w_0$ and $w_1$, and the other
component contains $z_0$ and $z_1$.
\end{define}

Given a decorated knot $(K, P)$, Ozsv\'ath--Szab\'o~\cite{OSKnots}, and
independently Rasmussen~\cite{RasmussenKnots}, defined a bigraded
$\bF_2$-vector space 
\[
\HFKh(K,P) = \bigoplus_{i, j \in \Z} \HFKh_i(K, P, j), 
\]
called the \emph{knot Floer homology} of
$(K, P)$. Thurston and the first author~\cite{JTNaturality}*{Theorem~1.8} showed that this is
natural with respect to the mapping class group of $(S^3, K, P)$.
By the work of Sarkar~\cite{SarkarMovingBasepoints}*{Theorem~1.1},
if we apply a positive Dehn twist $r$ along $K$ (see Definition~\ref{def:rolling}), the induced map
\[
r_* \colon \HFKh(K, P) \to \HFKh(K, P)
\]
is given by the formula
\[
r_* = \Id_{\HFKh(K,P)} + \Phi\Psi,
\]
where $\Phi$ and $\Psi$ are commuting endomorphisms of $\HFKh(K, P)$
such that $\Phi^2 = \Psi^2 = 0$.

To a decorated concordance $(C,\sigma)$ from $(K_0,P_0)$ to $(K_1,P_1)$,
the first author \cite{JCob} assigned a linear map
\[
F_{C,\sigma}\colon \hat{\HFK}(K_0,P_0)\to \hat{\HFK}(K_1,P_1).
\]
According to \cite{JMConcordance}, this is always non-vanishing.
Given a decorated knot $(K,P)$ in $S^3$ and a slice disk $D$ for $K$ in $D^4$,
Marengon and the first author~\cite{JMConcordance} defined an element
\[
t_{D,P} \in \HFKh_0(K,P,0)
\]
that is invariant under smooth isotopies of $D$. It is obtained by removing a ball
from $D^4$ about a point of $D$ that intersects $D$ in a disk,
choosing an arbitrary decoration $\sigma$ on the resulting
concordance $C$ from the decorated unknot $(U, P_U)$ to $(K, P)$,
and setting $t_{D,P} = F_{(C,\sigma)}(1)$, where
\[
F_{(C,\sigma)} \colon \HFKh(U, P_U) \cong \bF_2 \to \HFKh(K,P)
\]
is the induced map on knot Floer homology. The element $t_{D,P}$ is independent
of the choice of decoration $\sigma$.

In fact, $t_{D, P}$ is invariant under stable isotopy. Indeed, we have the following:

\begin{lem}
  Let $D$ be a slice disk of the based knot $(K, P)$ in $S^3$.
  If $S$ is a 2-knot in $D^4$, then
  \[
  t_{D, P} = t_{D \# S, P}.
  \]
\end{lem}

\begin{proof}
  Remove a ball from $D^4$ about a point of $S$ that intersects $D \# S$
  in a disk, and choose decorations compatible with $P_U$ and $P$.
  Then the resulting decorated concordance from $(U, P_U)$ to $(K, P)$ is the composition $\cC \circ \cC_S$, where
  $\cC = (C, \sigma)$ is a decorated concordance from $(U, P_U)$ to $(K, P)$,
  and $\cC_S = (C_S, \sigma_S)$ is a decorated concordance from $(U, P_U)$ to itself with $|\sigma_S| = 2$.
  The map
  \[
  F_{\cC_S} \colon \HFKh(U, P_U) \to \HFKh(U, P_U)
  \]
  is non-vanishing by \cite[Corollary~1.3]{JMConcordance}, and $\HFKh(U, P_U) \cong \bF_2$, hence
  \[
  t_{D \# S, P} = F_{\cC \circ \cC_S}(1) =
  F_{\cC} \circ F_{\cC_S}(1) = F_\cC(1) = t_{D, P},
  \]
  as claimed.
\end{proof}

With the help of our work joint with Ghiggini~\cite{GJOrientations}
on constructing canonical orientation systems in sutured Floer homology,
we will be able to lift the above constructions and results from $\bF_2$ to $\Z$ coefficients.

\section{Invariants of deform-spun slice disks}

It has been an open problem \cite{JMConcordance}*{Question~1.4}
whether there is a decorated knot $(K, P)$ with two slice disks $D$ and $D'$ such that $t_{D,P} \neq t_{D',P}$.
In this section, we give an affirmative answer to this question, by giving a formula for the invariant
of a deform-spun slice disk. Before we state our theorem, we recall some notation from \cite[Section~5.1]{ZConnected}.

Let $(K, P)$ and $(K', P')$ be decorated knots in $S^3$, and pick points $p \in K \setminus P$
and $p' \in K' \setminus P'$. Furthermore, we choose positive normal framings $(v_1, v_2, v_3)$ of $p$ and
$(v_1', v_2', v_3')$ of $p'$ such that $v_1$ is positively tangent to $K$ and $v_1'$ is positively
tangent to $K'$. We now construct a decorated concordance
\[
(W,\cF) \colon (-S^3, -K, P) \sqcup (S^3, K', P') \to (S^3, -K \# K', P)
\]
corresponding to taking the connected sum along $-p$ and $p'$; see Figure~\ref{fig:2}. The cobordism $W$ is
obtained by attaching a 4-dimensional 1-handle $D^1 \times D^3$ to $I \times (-S^3 \sqcup S^3)$
along $\{(1, p), (1, p')\}$. The underlying surface $F$ of $\cF = (F, \sigma)$ is obtained by attaching
the 2-dimensional 1-handle
\[
D^1 \times \{(t, 0, 0) : t \in [-1,1]\} \subset D^1 \times D^3
\]
to $I \times (-K \sqcup K)$ along $\{(1, p), (1, p')\}$.
Let $q$ be a point in the component of $K \setminus P$ not containing $p$,
and Let $q'$ be a point in the component of $K \setminus P'$ not containing $p'$.
Then the decoration $\sigma$ on $F$ is the union of $I \times \{p, p', q, q'\}$
and the core of the 2-dimensional 1-handle, and consists of three arcs.

Let us write
\[
G \colon \HFKh(-K, P) \otimes \HFKh(K', P') \to \HFKh(-K \# K', P)
\]
for the cobordism map induced by $(W, \cF)$.
According to \cite{ZConnected}*{Proposition~5.1}, the map $G$ is an isomorphism, whose
inverse $E$ is the cobordism map induced by $(W', \cF')$,
obtained by turning around and reversing the orientation of $(W,\cF)$.
Note that, given $-K \# K'$, the isomorphism $E$ depends on the connected sum sphere,
which is the belt sphere of the 1-handle described above.
We then have the following formula for the invariant of a slice disk obtained by deform-spinning.

\begin{thm}\label{thm:spinning}
  Let $K$ be a knot in $S^3$, and suppose that $B$ is an open 3-ball about a point of $K$ that intersects $K$
  in an unknotted arc.
  Let $D_{K, d}$ be a slice disk of the knot $-K \# K$, obtained by deform-spinning $K$
  using a diffeomorphism $d \in \Diff(S^3, K)$ that is the identity on $B$. If we write $V := \HFKh(K,P)$, then
  \[
  E(t_{D_{K, d}, P}) = d_* \in \Hom(V,V) \cong V^* \otimes V,
  \]
  where the isomorphism  $E \colon \HFKh(-K \# K, P) \to V^* \otimes V$ is described above.
\end{thm}

\begin{proof}
  Choose an isotopy $\phi$ of $S^3$ such that $\phi_0 = \Id_{S^3}$ and $\phi_1 = d$.
  By definition, $D_{K, d} = D_{\phi}$. Let $V_K$ be the link cobordism obtained from the
  cotrace cobordism $(I \times S^3, I \times K, \sigma)$ from $(\emptyset, \emptyset, \emptyset)$
  to $(-S^3, -K, P) \sqcup (S^3, K, P)$ by removing a ball about a point in $\sigma = I \times \{x, y\}$,
  where $x$ and $y$ are points from each component of $K \setminus P$.
  Furthermore, consider the concordance $(I \times S^3, C_\phi)$ from $(S^3, U)$ to $(S^3, K)$, obtained
  by removing a ball from $D^4$ about a point of the slice disk $D_{\phi}$.
  Choose an arbitrary decoration on $C_\phi$ consisting of two
  arcs, compatible with $P_U$ and $P$. Finally, let $T_{\phi}$ be the cobordism from $(-S^3, -K, P) \sqcup (S^3, K, P)$
  to itself that is the trace of the isotopy $\Id_{S^3} \sqcup \phi$. Then we have
  \begin{equation}\label{eqn:decomp}
  (W',\cF') \circ (S^3 \times I, C_\phi) = T_\phi \circ V_K.
  \end{equation}
  The left-hand side is obtained from $(W',\cF') \circ (D^4, D_{\phi})$,
  which is shown  schematically on the left of Figure~\ref{fig:2}, 
  by removing a 4-ball along the dividing set of $D_{\phi}$. 
  The right-hand side is obtained from the right of Figure~\ref{fig:2},
  the composition of the cotrace cobordism of $K$ with $T_\phi$,
  by removing a 4-ball along the dividing set of the cotrace cobordism.

\begin{figure}[ht!]
  \centering
\begingroup%
  \makeatletter%
  \providecommand\color[2][]{%
    \errmessage{(Inkscape) Color is used for the text in Inkscape, but the package 'color.sty' is not loaded}%
    \renewcommand\color[2][]{}%
  }%
  \providecommand\transparent[1]{%
    \errmessage{(Inkscape) Transparency is used (non-zero) for the text in Inkscape, but the package 'transparent.sty' is not loaded}%
    \renewcommand\transparent[1]{}%
  }%
  \providecommand\rotatebox[2]{#2}%
  \newcommand*\fsize{\dimexpr\f@size pt\relax}%
  \newcommand*\lineheight[1]{\fontsize{\fsize}{#1\fsize}\selectfont}%
  \ifx\svgwidth\undefined%
    \setlength{\unitlength}{376.6158973bp}%
    \ifx\svgscale\undefined%
      \relax%
    \else%
      \setlength{\unitlength}{\unitlength * \real{\svgscale}}%
    \fi%
  \else%
    \setlength{\unitlength}{\svgwidth}%
  \fi%
  \global\let\svgwidth\undefined%
  \global\let\svgscale\undefined%
  \makeatother%
  \begin{picture}(1,0.44082544)%
    \lineheight{1}%
    \setlength\tabcolsep{0pt}%
    \put(0,0){\includegraphics[width=\unitlength,page=1]{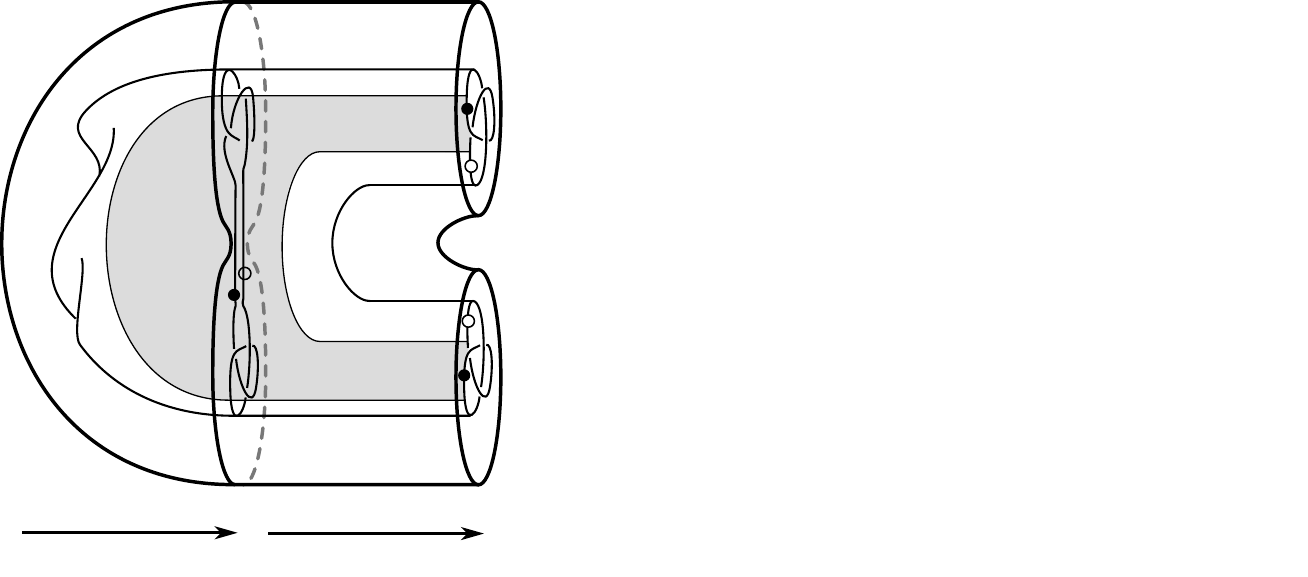}}%
    \put(0.04738912,0.00219367){\color[rgb]{0,0,0}\makebox(0,0)[lt]{\lineheight{0}\smash{\begin{tabular}[t]{l}$(D^4,D_\phi)$\end{tabular}}}}%
    \put(0.24223445,0.00219367){\color[rgb]{0,0,0}\makebox(0,0)[lt]{\lineheight{0}\smash{\begin{tabular}[t]{l}$(W',\cF')$\end{tabular}}}}%
    \put(0,0){\includegraphics[width=\unitlength,page=2]{fig1_new.pdf}}%
    \put(0.65478756,0.00219367){\color[rgb]{0,0,0}\makebox(0,0)[lt]{\lineheight{0}\smash{\begin{tabular}[t]{l}$(I \times S^3, I \times K, \sigma)$\end{tabular}}}}%
    \put(0.92941901,0.00219367){\color[rgb]{0,0,0}\makebox(0,0)[lt]{\lineheight{0}\smash{\begin{tabular}[t]{l}$T_\phi$\end{tabular}}}}%
    \put(0,0){\includegraphics[width=\unitlength,page=3]{fig1_new.pdf}}%
  \end{picture}%
\endgroup%

  \caption{On the left, the cobordism $(W',\cF') \circ (D^4, D_{\phi})$. The concordance
  $(S^3 \times I, C_\phi)$ is obtained from $(D^4, D_{\phi})$ by removing a
  4-ball along the dividing set of $D_{\phi}$. On the right, the cotrace cobordism of $K$
  composed with $T_\phi$, which gives rise to $V_K \circ T_\phi$ after removing a 4-ball.
  \label{fig:2}}
\end{figure}

  By \cite{JuhaszZemkeContactHandles}*{Theorem~1.1}, the link cobordism map
  \[
  F_{V_K} \colon \bF_2 \to V^* \otimes V \cong \Hom(V,V)
  \]
  maps $1 \in \bF_2$ to $\cotr_V(1) \in V^* \otimes V$,
  which is naturally identified with $\Id_V \in \Hom(V,V)$.
  Furthermore, $F_{T_\phi} = \Id_{V^*} \otimes d_*$.
  As $F_{C_\phi}(1) = t_{D_\phi,P}$, by applying the functor $\SFH$ to equation~\eqref{eqn:decomp},
  we obtain that
  \[
  E(t_{D_\phi,P}) = (\Id_{V^*} \otimes d_*) \circ F_{V_K}(1) =
  (\Id_{V^*} \otimes d_*) \circ \cotr_V(1).
  \]
  Now $(\Id_{V^*} \otimes d_*) \circ \cotr_V(1) \in V^* \otimes V$ is
  naturally identified with $d_* \in \Hom(V,V)$, and the result follows.
\end{proof}

\begin{cor}
  Let $(K, P)$ be a decorated knot in $S^3$, and $B$ an open 3-ball about a point of $K$ that intersects it in an unknotted arc.
  If $d$, $d' \in \Diff(S^3, K)$ are the identity on $B$,
  then $t_{D_{K, d}, P} = t_{D_{K, d'}, P}$ if and only if $d_* = d'_*$ on $\HFKh(K, P)$.
\end{cor}

\begin{prop}
  Let $(K, P)$ be a decorated knot in $S^3$, and let $D_{K,r}$ be the slice disk obtained
  by $1$-roll-spinning $K$. Then $t_{D_{K, \Id}} = t_{D_{K, r}}$ if and only if the basepoint moving
  map $\Id_{\HFKh(K,P)} + \Phi\Psi$, defined by Sarkar~\cite{SarkarMovingBasepoints}, is the identity.
\end{prop}

\begin{proof}
  By Theorem~\ref{thm:spinning}, we have $t_{D_{K, \Id}} = (\Id_{S^3})_* = \Id_V$ and $t_{D_{K, r}} = r_*$.
  According to Definition~\ref{def:rolling}, the diffeomorphism $r$ is a positive Dehn twist along $K$.
  Consequently, $r_*$ is the basepoint moving map.
  By Sarkar~\cite{SarkarMovingBasepoints}*{Theorem~1.1}, the latter is $\Id_V + \Phi \Psi$.
\end{proof}

According to \cite{SarkarMovingBasepoints}*{Section~6}, the map $\Phi\Psi$ is non-vanishing
for the majority of prime knots up to nine crossings.
For example, by \cite{SarkarMovingBasepoints}*{Theorem~6.1}, this holds for any alternating knot
$K$ such that
\[
t^{\frac{\sigma(K)}{2}} \Delta_K(t) \neq \frac{1+ t^{2m+1}}{1+t}
\]
for any integer $m$. Clearly, there are infinitely many such knots $K$.
Hence, in all these cases, $D_{K, \Id}$ and $D_{K, r}$ are stably non-isotopic slice disks of $K$.
Furthermore, they have diffeomorphic complements by Proposition~\ref{prop:complement}.

Note that $(\Phi\Psi)^2 = 0$. Hence, if the formula $\Id_V + \Phi \Psi$ also holds for
the basepoint moving map $r_*$ over the integers, then
\[
t_{D_{K, r^l}} = (\Id_V + \Phi \Psi)^l = \Id_V + l \cdot \Phi \Psi.
\]
Consequently, all the slice disks $D_{K, r^l}$ for $l \in \Z$ would be pairwise stably non-isotopic.
Hence, with this caveat, we have obtained the following result:

\begin{thm}\label{thm:rolling}
  Assume Sarkar's basepoint moving formula holds over the integers, and let $K$ be an alternating knot such that
  \[
  t^{\frac{\sigma(K)}{2}} \Delta_K(t) \neq \frac{1+ t^{2m+1}}{1+t}
  \]
  for any integer $m$. Then the slice disks $D_{K, r^l}$ for $l \in \Z$, obtained by roll-spinning, are
  distinguished by $t_{D_{K, r^l}}$ up to stable isotopy, but have diffeomorphic complements.
\end{thm}

Note that, if we work with $\bF_2$ coefficients, then we are only able to distinguish
$D_{K, r^l}$ for $l$ even and odd.

\section{Invariants of slice disks arising from concordances, and their rank}

We have the following straightforward generalization of Theorem~\ref{thm:spinning} to slice disks
arising from concordances as in Definition~\ref{def:conc}:

\begin{thm}\label{thm:concordance}
  Let $(K, P)$ and $(K', P')$ be decorated knots in $S^3$ that both pass through a point $x \in S^3$,
  and let $\cC = (C,\sigma)$ be a decorated concordance from $(K, P)$ to $(K', P')$, as in Definition~\ref{def:conc}.
  If we write $V := \HFKh(K, P)$ and $V' := \HFKh(K', P')$, then
  \[
  E(t_{D_{\cC}, P}) = F_{\cC} \in \Hom(V, V') \cong V^* \otimes V',
  \]
  where the isomorphism $E \colon \HFKh(-K \# K', P) \to V^* \otimes V'$ is described above.
\end{thm}

Consequently, if $\cC$ and $\cC'$ are decorated concordances from $(K, P)$ to $(K', P')$, then
$t_{D_{\cC}} = t_{D_{\cC'}}$ if and only if $F_{\cC} = F_{\cC'}$.

\begin{cor}\label{cor:split}
  Let $(K, P)$ and $(K', P')$ be decorated knots in $S^3$ that both pass through a point $x \in S^3$.
  If a slice disk $D$ of $-K \# K'$ is the boundary connected sum of slice disks
  of $K$ and $K'$, then $t_D \in \Hom(V,V')$ has rank one.
\end{cor}

\begin{proof}
  By Lemma~\ref{lem:sum}, there exists a decorated concordance $\cC$ from $(K, P)$ to $(K', P')$
  such that $D_\cC = D$. If $D$ is a boundary connected sum, then $\cC$ is the product of a concordance
  from $(K, P)$ to a decorated unknot $(U, P_U)$ with $|P_U| = 2$, and a concordance from $(U, P_U)$ to $(K', P')$.
  In particular, the concordance map $F_\cC$ factors through $\HFKh(U, P_U)$, which has rank one.
  Since $F_\cC$ is non-vanishing by \cite[Corollary~1.3]{JMConcordance}, it has rank one.
  The result now follows from Theorem~\ref{thm:concordance}.
\end{proof}

\begin{example}\label{ex:sum}
  This example illustrates that the rank of $t_D$ depends on the specific choice of connected sum sphere
  used in the identification of $\HFKh(-K \# K', P)$ with $\Hom(V, V')$. We construct a spun slice disk
  that is the boundary connected sum of slice disks; see Figure~\ref{fig:3}. Let $K_0$ and $K_1$ be knots,
  and set $K = -K_0 \# K_1$. Choose a decoration $P$ on $K$.
  Let $S$ be the connected sum sphere, write $S \cap K = \{p, q\}$, and let $B$ be
  a ball about $p$. We write $D$ for the slice disk obtained by spinning the arc $K \setminus B$.
  Then $D$ is a slice disk for $-K \# K = (K_0 \# {-}K_1) \# (-K_0 \# K_1)$. Using the connected sum sphere $\d B$
  for $-K \# K$, we obtain an identification of $\HFKh(-K \# K, P)$ with $V^* \otimes V$, where $V = \HFKh(K, P)$,
  such that $t_D$ corresponds to $\Id_V$, and hence has rank $\dim(V)$.

\begin{figure}
  \centering
  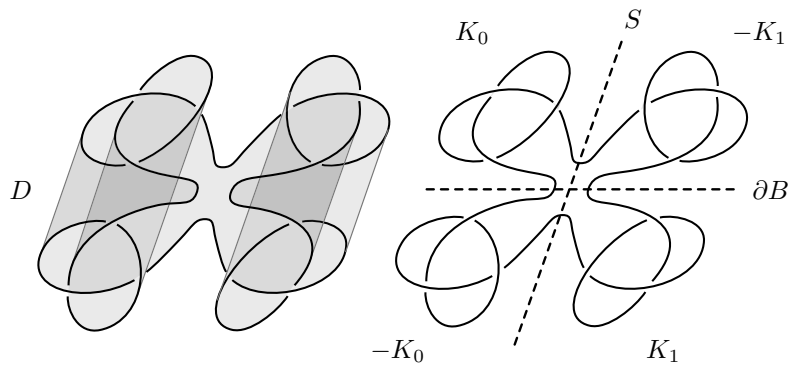
  \caption{The slice disk $D$ is obtained by spinning $-K_0 \# K_1$ about $\d B$.
  It splits as a boundary connected sum along the sphere $S$.
  \label{fig:3}}
\end{figure}

  On the other hand, we can also split $-K \# K$
  along the 2-sphere $S$ through $\{-q, q\}$, and $D$ decomposes as a boundary connected sum of the spun slice disks
  for $K_0 \# {-}K_0$ and $-K_1 \# K_1$ in this direction.
  Hence, $S$ gives an identification of $\HFKh(-K \# K, P)$ with $V_0^* \otimes V_1$,
  where $V_0 = \HFKh(-K_0 \# K_0, P_0)$ and $V_1 = \HFKh(-K_1 \# K_1, P_1)$, with respect to which $t_D$ has rank one
  by Corollary~\ref{cor:split}. When $K_0$ and $K_1$ are equivalent knots, then both $V_0$ and $V_1$ are isomorphic to $V$.
\end{example}

\begin{define}
  Let $(K, P)$ and $(K', P')$ be decorated knots in $S^3$,
  and consider a slice disk $D$ of the composite knot $(-K \# K', P)$
  with connected sum sphere $S$. Then we define the \emph{rank} $\rk_S(D)$ of $D$ with respect to $S$
  to be the rank of $t_{D, P}$, viewed as an element of $\Hom(V, V')$, where
  $V = \HFKh(K, P)$ and $V' = \HFKh(K', P')$, and we identify $\HFKh(-K \# K', P)$
  with $\Hom(V, V')$ using the connected sum sphere $S$.

  By Lemma~\ref{lem:sum}, there is a concordance $\cC$ from $K$ to $K'$ such that $D = D_\cC$.
  According to Theorem~\ref{thm:concordance}, $E(t_{D_\cC, P}) = F_\cC$,
  which preserves the Alexander and Maslov gradings by \cite[Theorem~1.2]{JMConcordance}.
  We can hence refine $\rk_S(D)$ by setting
  \[
  \rk_{S, j}(D, i) = \rk\left(F_\cC|_{\HFKh_j(K, i)}\right) \text{ and }
  \rk_S(D, i) = \rk\left(F_\cC|_{\HFKh(K, i)}\right).
  \]
\end{define}

\begin{lem}\label{lem:composite}
  Let $(K, P)$ and $(K', P')$ be decorated \emph{prime} knots in $S^3$ such that $K' \neq -K$.
  If $D$ is a slice disk for $-K \# K'$, then $\rk_S(D)$, $\rk_{S, j}(D, i)$, and
  $\rk_S(D, i)$ are independent of the choice of connected sum sphere $S$.
\end{lem}

In this case, we write $\rk(D)$, $\rk_j(D, i)$, and $\rk(D, i)$ for $\rk_S(D)$, $\rk_{S, j}(D, i)$,
and $\rk_S(D, i)$, respectively. These are clearly invariants of $D$ up to stable diffeomorphism.

\begin{proof}
  Since $K$ and $K'$ are prime, there is an isotopically unique connected sum sphere $S$ for $-K \# K'$
  by the uniqueness of prime factorization of knots. Indeed, if $S'$ is another connected sum sphere,
  then we can make it transverse to $S$, and remove innermost components of $S^3 \setminus (S \cup S')$
  by isotopies through connected sum spheres until $S$ and $S'$ become parallel.
  Furthermore, as $K' \neq -K$ and $K'$ has to lie on the positive side of both $S$ and $S'$,
  the spheres $S$ and $S'$ are oriented coherently.

  As $-K \# K'$ is a decorated knot, the ambient isotopy taking $S'$ to $S$ might move the decoration $P$
  around $-K \# K'$ a number of times. We now show that the element $E(t_{D, P}) \in \Hom(V, V')$
  is invariant under moving $P$ around $K \# K'$.
  This corresponds to winding $\sigma$ around the boundary component $-K \# K'$ of the surface $F'$
  in the connected sum cobordism $(W', \cF')$.
  Recall that $t_{D, P}$ is defined by removing a ball from $D^4$ about a point of $D$,
  choosing a decoration $\nu$ on the resulting concordance $C$ from $(U, P_U)$ to $(-K \# K', P)$,
  and setting $t_{D, P} = F_{(C, \nu)}(1)$. By the functoriality of the decorated link cobordism maps
  in link Floer homology,
  \[
  E(t_{D, P}) = F_{(W', \cF')} \circ F_{(C, \nu)}(1) = F_{(W', \cF') \circ (C, \nu)}(1);
  \]
  see Figure~\ref{fig:2}.
  Let $\sigma'$ be another decoration on $F'$ that differs from $\sigma$ by winding it around $-K \# K'$.
  Then, for a suitable choice of decoration $\nu'$ on $C$, the decorations $\sigma \cup \nu$
  and $\sigma' \cup \nu'$ are isotopic in $C \cup F'$ relative to their boundaries. As $t_{D, P}$ is independent
  of the choice of $\nu$, the result follows.

  We remark that we could even wind $\sigma$ around the boundary components $-K$ and $K'$ of $F'$
  without changing the rank of $D$. Indeed, the induced map would act on the two factors of $V^* \otimes V'$
  by basepoint moving automorphisms of $V^*$ and $V'$, which preserves the rank of each element.
\end{proof}

Recall that a concordance is invertible if it has a left inverse in the cobordism
category of links; see Sumners~\cite{Sumners}.

\begin{cor}\label{cor:rank}
  Let $(K, P)$ and $(K', P')$ be decorated, prime, slice knots in $S^3$, with slice disks $D$ and $D'$,
  such that $K' \neq -K$ and $K \neq U$. Furthermore, let $\cC$ be an invertible concordance from $K$ to $K'$.
  Then the slice disk $D_\cC$ of $-K \# K'$ and the boundary connected sum $-D \natural D'$ have different ranks,
  and are hence not stably diffeomorphic.
\end{cor}

\begin{proof}
  By Corollary~\ref{cor:split}, we have $\rk(-D \natural D') = 1$. On the other hand, $t_{D_\cC}$
  corresponds to $F_\cC$ by Theorem~\ref{thm:concordance}. As $\cC$ is invertible, the map $F_\cC$
  is injective, and hence has rank $\dim(V)$; see \cite{JMConcordance}.
  But $K \neq U$, so $\dim(V) > 1$ by the genus detection of knot Floer homology.
  It follows that $\rk(D_\cC) > 1$, and hence $D_\cC$ and $-D \natural D'$ are not stably diffeomorphic.
\end{proof}

For example, in the above result, we could take $K = K'$
to be any nontrivial, prime, achiral, slice knot, and $\cC = I \times K$,
in which case $D_\cC$ is obtained by spinning $K$. It follows that this spun
slice disk is not a boundary connected sum. Compare this with Example~\ref{ex:sum},
which shows that this fails when $K$ is composite.
Note that one can also use the fundamental group to show that $D_\cC$ in Corollary~\ref{cor:rank}
is not stably diffeomorphic to a boundary connected sum. The basic idea is that the map
$\pi_1(K) \to \pi_1(K) \ast_\mu \pi_1(K') \to \pi_1(D_{\cC})$ is injective as $\cC$ is invertible,
while the longitude of $K$ maps to zero under the composition
$\pi_1(K) \to \pi_1(K) \ast_\mu \pi_1(K') \to \pi_1(-D \natural D')$; cf.~Remark~\ref{rem:section} below.

More generally, the rank gives a lower bound on the complexity of certain sections of
the slice disk, in the following sense.

\begin{thm}\label{thm:section}
  Let $(K, P)$ and $(K', P')$ be decorated knots in $S^3$, and let $D$ be a slice disk
  of $(-K \# K', P)$ with connected sum sphere $S$. Suppose that there is a properly embedded
  3-ball $H$ in $D^4$ transverse to $D$, and such that $S^3 \cap H = S$.
  Consider the link $L$ in $S^3$ obtained by capping off $(H, H \cap D)$ with the trivial tangle $(D^3, D^1)$. Then
  \[
  \rk_S(D) \le \dim \left(\HFKh(L)\right).
  \]
  Analogous inequalities hold for $\rk_{S, j}(D, i)$ and $\rk_S(D, i)$. Hence, when $L$ is a knot,
  \[
  \max \{\, i \in \Z : \rk_S(D , i) \neq 0 \,\} \le g(L).
  \]
\end{thm}

\begin{proof}
  By Lemma~\ref{lem:sum}, there is a decorated concordance $\cC = (C, \sigma)$ in $I \times S^3$ such that $D = D_\cC$.
  By construction, $H \cup D^3$ is an embedded $S^3$ in $I \times S^3$ that intersects $C$ in the link $L$.
  We can isotope $\sigma$ on $C$ such that it intersects each component of $L$ in exactly two points.
  Hence, the link concordance $\cC$ factorizes through the decorated link $(H \cup D^3, L, \sigma \cap L)$.
  The result follows by applying the link Floer homology functor. The grading refined statement follows
  from the fact that the concordance maps preserve the Alexander and Maslov gradings.
  The last inequality holds since knot Floer homology detects the Seifert genus; see Ozsv\'ath--Szab\'o~\cite{genusbounds}.
\end{proof}

\begin{rem}\label{rem:section}
  As in Corollary~\ref{cor:rank}, we can consider the special case of Theorem~\ref{thm:section}
  when $D = D_\cC$ for an invertible concordance $\cC$. Then $F_\cC$ is injective, so we obtain that
  \[
  \rk_{S, j}(D, i) = \dim \left(\HFKh_j(K, i)\right).
  \]
  In particular, $g(K) \le g(L)$.
  However, the section of $\cC$ from $K$ to $L$ is also invertible, hence the genus inequality also follows from
  \cite[Corollary~1.7]{JMConcordance}, which can be shown without using knot Floer homology,
  instead invoking Gabai's higher genus generalization of Dehn's lemma \cite[Corollary~6.23]{Gabai}.
  When $L$ is the unknot, we recover the last part of Corollary~\ref{cor:rank}, stating that $D_\cC$ is not
  a boundary connected sum.
\end{rem}

\bibliographystyle{custom}
\bibliography{biblio}

\end{document}